\documentclass[12pt,reqno]{amsart}
\usepackage[T1]{fontenc}
\usepackage{amsmath,amssymb,amsthm,mathrsfs}
\usepackage{newtxtext,newtxmath}
\usepackage{microtype}
\usepackage[colorlinks=true,linkcolor=blue,citecolor=blue,urlcolor=blue]{hyperref}
\numberwithin{equation}{section}
\newtheorem{theorem}{Theorem}[section]
\newtheorem{proposition}[theorem]{Proposition}
\newtheorem{lemma}[theorem]{Lemma}

\theoremstyle{definition}

\theoremstyle{remark}

\DeclareMathOperator{\vspan}{span}

\title[Stability near sums of fractional ground states]{Sharp stability near sums of ground states\newline for fractional Schr\"odinger equations}
\author{Hua Chen}
\address{School of Mathematics and Statistics, Wuhan University, Wuhan, 430072, China}
\email{chenhua@whu.edu.cn}
\author{Yun-Lu Fan}
\address{School of Mathematics and Statistics, Wuhan University, Wuhan, 430072, China}
\email{yunlufan@whu.edu.cn}
\author{Xin Liao}
\address{MOE-LCSM, School of Mathematics and Statistics, Hunan Normal University, Changsha, 410081, 
China}
\email{xin\_liao@whu.edu.cn}
\date{}

\begin{document}
\begin{abstract}
We study quantitative stability for the fractional Schr\"odinger equation
$(-\Delta)^s u+u-|u|^\alpha u=0$ near finite sums of widely separated
positive ground states. For every $n\ge1$, $0<s<1$, and Sobolev-subcritical
exponent $\alpha>0$, we estimate the $H^s$ distance to the family of such
sums in terms of the $H^{-s}$ norm of the equation's residual. The optimal
rate changes at $\alpha=n/[2(n+2s)]$, with a logarithmic correction at
the threshold. Above the threshold the rate is
$t^{(n+2s)/(n+2s+1)}$; below it the exponent is
$[(1+\alpha)(n+2s)-n/2]/(n+2s+1)$.
The proof combines uniform invertibility away from the translation modes
with precise interaction estimates and a weighted correction of the
approximate configuration. This correction resolves the translation
interactions even when the nonlinearity has a small exponent.
We also construct nonnegative configurations that attain these rates.
The estimates quantify the effect of the algebraic decay of fractional
ground states on stability.
\end{abstract}
\maketitle

\section{Introduction}
Let $n\ge1$, $0<s<1$, and
\begin{equation}\label{eq:crit_exp}
  \begin{cases}
    0<\alpha<\dfrac{4s}{n-2s},&n>2s,\\[1mm]
    \alpha>0,&n\le2s.
  \end{cases}
\end{equation}
We consider the stationary fractional Schr\"odinger equation
\begin{equation}\label{eq:problem}
  (-\Delta)^s u+u-|u|^\alpha u=0
  \quad\text{in }\mathbb R^n.
\end{equation}
Our aim is to determine how the residual of this equation controls the
distance to a finite sum of ground states when their centers are far
apart. Throughout the paper, functions are real-valued, and the
fractional Laplacian has Fourier symbol $|\xi|^{2s}$.

The equation has a positive ground state $Q$, unique up to translation.
We take $Q$ to be radial about the origin. Its existence and qualitative
properties are discussed in \cite{FQT12}; uniqueness and nondegeneracy
were proved by Frank and Lenzmann \cite{FL13} in dimension one and by
Frank, Lenzmann and Silvestre \cite{FLS16} in arbitrary dimension.
In particular,
\begin{equation}\label{eq:ground-state}
  (-\Delta)^s Q+Q=Q^{1+\alpha},\qquad Q>0,
\end{equation}
and
\begin{equation}\label{eq:nondegeneracy}
  \ker\bigl(( -\Delta)^s+1-(1+\alpha)Q^\alpha\bigr)
  =\vspan\{\partial_{x_1}Q,\ldots,\partial_{x_n}Q\}.
\end{equation}
Unlike the ground state of the local massive equation, $Q$ has an
algebraic tail:
\begin{equation}\label{eq:ground-state-decay}
  Q(x)\simeq(1+|x|)^{-n-2s}.
\end{equation}
As a consequence, separated profiles interact at polynomial scales.
The relation between these interactions and their translation
projections determines the stability estimates below.

We use the energy norm
\begin{equation}\label{eq:energy-norm}
  \|v\|_{H^s}^2
  :=\|v\|_{L^2}^2+\|(-\Delta)^{s/2}v\|_{L^2}^2,
\end{equation}
and write $H^{-s}$ for its dual under the $L^2$ pairing. Set
\begin{equation}\label{eq:defect}
  L:=(-\Delta)^s+1,\qquad
  \mathcal E(u):=Lu-|u|^\alpha u,\qquad
  \Gamma(u):=\|\mathcal E(u)\|_{H^{-s}}.
\end{equation}
The restriction \eqref{eq:crit_exp} ensures that
$\mathcal E:H^s\to H^{-s}$ is well defined. With the norm
\eqref{eq:energy-norm}, $L:H^s\to H^{-s}$ is an isometric isomorphism.

Fix $m\ge2$. For translation parameters $y_1,\ldots,y_m\in\mathbb R^n$,
write
\begin{equation}\label{eq:configuration}
  Q_i(x):=Q(x+y_i),\qquad
  \sigma:=\sum_{i=1}^m Q_i,\qquad
  R:=\min_{i\ne j}|y_i-y_j|.
\end{equation}
Thus the physical centers are $-y_i$. The distance to the family of all such sums is
\begin{equation}\label{eq:distance}
  d_m(u):=\inf_{y_1,\ldots,y_m\in\mathbb R^n}
       \left\|u-\sum_{i=1}^mQ(\cdot+y_i)\right\|_{H^s}.
\end{equation}
The stability modulus has three regimes. Define
\begin{equation}\label{eq:stability-modulus}
  \alpha_*:=\frac{n}{2(n+2s)},\qquad
  \Phi(t):=
  \begin{cases}
    t^{\frac{n+2s}{n+2s+1}},&\alpha>\alpha_*,\\[1mm]
    t^{\frac{n+2s}{n+2s+1}}(1+|\log t|)^{1/2},&\alpha=\alpha_*,\\[1mm]
    t^{\frac{(1+\alpha)(n+2s)-n/2}{n+2s+1}},&0<\alpha<\alpha_*.
  \end{cases}
\end{equation}
Here $0<t\le1$, and $\Phi(0):=0$.

\begin{theorem}[Quantitative stability]\label{thm-quantitative}
Assume \eqref{eq:crit_exp} and fix $m\ge2$. There exist
$\varepsilon_0>0$, $R_0>1$, and $C>0$, depending only on
$n,s,\alpha,m$, with the following property. If $u\in H^s(\mathbb R^n)$
and some $\widetilde y_1,\ldots,\widetilde y_m\in\mathbb R^n$ satisfy
\[
  \left\|u-\sum_{i=1}^m Q(\cdot+\widetilde y_i)\right\|_{H^s}
       \le\varepsilon_0,
  \qquad
  \min_{i\ne j}|\widetilde y_i-\widetilde y_j|\ge2R_0,
\]
then there are parameters $y_1,\ldots,y_m$ with separation $R\ge R_0$
such that, for $\sigma$ as in \eqref{eq:configuration},
\begin{equation}\label{stable}
  \|u-\sigma\|_{H^s}\le C\Phi(\Gamma(u)),\qquad
  R^{-n-2s-1}\le C\Gamma(u).
\end{equation}
The parameters may be chosen so that
\begin{equation}\label{eq:modulation}
  \langle u-\sigma,\partial_{x_j}Q_i\rangle_{H^s}=0
  \quad(1\le i\le m,\ 1\le j\le n).
\end{equation}
In particular, $d_m(u)\le C\Phi(\Gamma(u))$.
\end{theorem}

By decreasing $\varepsilon_0$ and increasing $R_0$, the hypotheses
ensure that $\Gamma(u)\le1$, so that \eqref{eq:stability-modulus} applies.
For one profile, nondegeneracy instead gives the linear estimate
$d_1(u)\lesssim\Gamma(u)$ in a neighborhood of the family of translates
of $Q$. The nonlinear moduli in Theorem~\ref{thm-quantitative} concern
the interaction of at least two profiles.

\begin{theorem}[Optimality]\label{sharp}
Assume \eqref{eq:crit_exp} and fix $m\ge2$. There are nonnegative
functions $u_R\in H^s(\mathbb R^n)$, indexed by sufficiently large $R$,
and parameters $y_{1,R},\ldots,y_{m,R}$ with minimum separation $R$,
such that
\[
  \left\|u_R-\sum_{i=1}^mQ(\cdot+y_{i,R})\right\|_{H^s}\longrightarrow0,
  \qquad \Gamma(u_R)\longrightarrow0,
\]
and
\begin{equation}\label{eq:sharpness-main}
  d_m(u_R)\ge c\Phi(\Gamma(u_R))
\end{equation}
for some $c>0$ independent of $R$. Thus each rate in
\eqref{eq:stability-modulus} is optimal, including the logarithmic factor
at $\alpha=\alpha_*$.
\end{theorem}

\subsection{Related results}
Quantitative stability asks whether the error in an inequality or an
equation controls distance to the corresponding set of optimizers or
solutions. For the Sobolev inequality, the theorem of Bianchi and
Egnell \cite{BE1991} bounds the squared distance to the optimizer
manifold by the Sobolev deficit. Stability of the Euler--Lagrange
equation involves a different quantity, namely the norm of its
residual. In the critical problem, Struwe's compactness theorem
\cite{Struwe84} identifies finite sums of bubbles as the relevant
limiting configurations.

Quantitative versions of this description were obtained by Ciraolo,
Figalli and Maggi \cite{CFM18}, Figalli and Glaudo \cite{FG20}, and
Deng, Sun and Wei \cite{Wei}. In the multi-bubble setting, the sharp
rate depends on the dimension, and a logarithmic correction occurs
at the transition. Fractional analogues were developed by De Nitti
and K\"onig \cite{NK23} and Aryan \cite{Aryan23}. These results concern
the critical, scale-invariant equation. In \eqref{eq:problem}, the mass
term fixes the scale of the ground state, leaving translations as the
only parameters.

For the local massive equation, Chen, Fan and Liao \cite{CFL25}
established sharp stability near sums of ground states. Here the algebraic
tail in \eqref{eq:ground-state-decay} separates the scale of the interaction
error from that of its translation projections by a power of the separation.
This yields a sublinear stability modulus even above the integrability
threshold, while the nonintegrable and borderline regimes produce a
different power and a logarithmic correction, respectively.
Weighted linear estimates and projected nonlinear equations are established
tools in the construction of fractional standing waves; see D\'avila,
del Pino and Wei \cite{DPW14}. Our use of these tools is to control a given
approximate solution and to recover the sharp residual scale uniformly
over separated configurations, including arbitrarily small admissible
exponents. The result is local to the family of separated positive profiles;
it does not assert a global decomposition for arbitrary bounded-residual
sequences or stability for arbitrary relative signs or complex phases.

\subsection{Outline of the proof}
The interaction error is
\begin{equation}\label{eq:interaction-error}
  f:=\sigma^{1+\alpha}-\sum_{i=1}^m Q_i^{1+\alpha},
  \qquad \mathcal E(\sigma)=-f.
\end{equation}
Its positivity and homogeneity imply $|\nabla f|\lesssim f$.
Interpolation and duality therefore give the useful equivalence
\[
  \|f\|_{H^{-s}}\simeq\|f\|_{L^2}\simeq\|f\|_{H^s}.
\]
For configurations with minimum separation $R$, the interaction estimates yield
\begin{equation}\label{eq:interaction-scales}
  \|f\|_{H^{-s}}\simeq
  \begin{cases}
    R^{-n-2s},&\alpha>\alpha_*,\\
    R^{-n-2s}(\log R)^{\frac{1}{2}},&\alpha=\alpha_*,\\
    R^{\frac{n}{2}-(1+\alpha)(n+2s)},&0<\alpha<\alpha_*.
  \end{cases}
\end{equation}
The threshold is exactly the condition $Q^\alpha\in L^2$.
Below it, the region between the centers contributes to the leading
interaction; at equality, the relevant radial integral is logarithmic.

The translation projections have the smaller scale
$R^{-n-2s-1}$. Radial symmetry cancels the constant part of a distant
profile, so that its first spatial variation supplies the leading
term. A geometric argument shows that these terms cannot cancel at
every center. To transfer this lower bound to a general nearby function,
we first choose its parameters so that $u=\sigma+\rho$ satisfies
\eqref{eq:modulation}.

Let $\mathscr F$ be the span of the translation modes
$\partial_{x_j}Q_i$, and let $P_{\mathscr F^\perp}$ denote the
$H^s$ orthogonal projection. Uniform invertibility on
$\mathscr F^\perp$ controls $\rho$ by $\Gamma(u)+\|f\|_{H^{-s}}$.
For small $\alpha$, however, this norm estimate alone does not make
the projected error negligible compared with $R^{-n-2s-1}$.
We therefore construct a correction $w_\sigma\in\mathscr F^\perp$
such that
\begin{equation}\label{eq:corrected-configuration}
  U_\sigma:=\sigma+w_\sigma,\qquad
  P_{\mathscr F^\perp}L^{-1}\mathcal E(U_\sigma)=0.
\end{equation}
The correction satisfies $\|w_\sigma\|_{H^s}\lesssim\|f\|_{H^{-s}}$.
For $\alpha<\alpha_*$ we also establish a weighted pointwise bound;
for $\alpha\ge\alpha_*$ the energy estimate and refined translation
projections suffice. Together these estimates give, for every admissible exponent,
\[
  \langle\mathcal E(U_\sigma),\nabla Q_i\rangle
  =-\int_{\mathbb R^n}f\nabla Q_i\,dx+o(R^{-n-2s-1}).
\]
The finite-dimensional residual is retained in
\eqref{eq:corrected-configuration}; the correction removes only its
orthogonal component.

For later use, write
\begin{equation}\label{eq:nonlinear-remainder}
  N_\sigma(w):=|\sigma+w|^\alpha(\sigma+w)
            -\sigma^{1+\alpha}-(1+\alpha)\sigma^\alpha w.
\end{equation}
Subtracting the projected equations for $\rho$ and $w_\sigma$ cancels
$f$. Uniform invertibility and a nonlinear difference estimate then
give $\|u-U_\sigma\|_{H^s}\lesssim\Gamma(u)$.
Consequently,
\[
  R^{-n-2s-1}\lesssim\Gamma(U_\sigma)
                  \lesssim\Gamma(u).
\]
Combining this with \eqref{eq:interaction-scales} proves
Theorem~\ref{thm-quantitative}. Corrected configurations also provide
the examples in Theorem~\ref{sharp}.

Section~\ref{sec2} establishes the interaction estimates, including the
translation projections. Section~\ref{sec3} develops the projected
linear theory and proves the stability estimate.
Section~\ref{sec4} constructs the examples proving optimality.

\subsection{Notation}
All integrals without a specified domain are over $\mathbb R^n$.
The symbols $\lesssim$ and $\gtrsim$ denote upper and lower bounds up to positive
constants depending only on the fixed parameters $n,s,\alpha,m$,
unless stated otherwise. These constants are independent of the
translation parameters; $\simeq$ means that both bounds hold.
We write $C$ and $c$ for positive constants whose values may change
between estimates, and use $\log$ for the natural logarithm.
The notation $o_R(1)$ denotes a quantity
tending to zero as the minimum separation $R$ tends to infinity,
uniformly over the configurations under consideration. Duality between
$H^{-s}$ and $H^s$ is denoted by $\langle\cdot,\cdot\rangle$; the
subscript $H^s$ specifies the energy inner product.
We write $B_r(z)$ for the open ball of radius $r$ centered at $z$,
with $B_r=B_r(0)$, and $e_1$ for the first coordinate unit vector.
The positive and negative parts are $v^\pm=\max\{\pm v,0\}$.

\section{Interaction estimates}\label{sec2}
We begin with the interaction error $f$ defined in
\eqref{eq:interaction-error}. The estimates below are uniform over all
configurations with a prescribed minimum separation. Unless stated
otherwise, their constants may depend on $m,n,s$, and $\alpha$.

We first record the derivative estimates needed below. They follow from
the eigenfunction asymptotics in \cite[Lemma C.2]{FLS16}, after a change
of dimension; the argument also covers $n=1$.

\begin{lemma}\label{lem:ground-state-gradient}
There exists $c_Q>0$ such that
\begin{equation}\label{eq:ground-state-gradient}
    |\nabla Q(x)|\lesssim(1+|x|)^{-n-2s-1},~~
    \nabla Q(x)
    =-c_Q(n+2s)\frac{x}{|x|^{n+2s+2}}
       +o(|x|^{-n-2s-1})
\end{equation}
as $|x|\to\infty$. Moreover,
$Q(x)=c_Q|x|^{-n-2s}+o(|x|^{-n-2s})$.
\end{lemma}
\begin{proof}
The regularity theory in \cite{FL13,FLS16} gives
$Q\in H^{2s+1}(\mathbb R^n)\cap C^\infty(\mathbb R^n)$.
Set $\psi(r)=-Q'(r)/r$ and regard this profile as a radial function on
$\mathbb R^{n+2}$. It extends smoothly across the origin, is nonnegative
and nonzero, and belongs to $L^2(\mathbb R^{n+2})$, since
\[
    \int_0^\infty|\psi(r)|^2r^{n+1}\,dr
    =\int_0^\infty|Q'(r)|^2r^{n-1}\,dr<\infty.
\]
For a smooth radial profile $a$, direct differentiation gives
\[
    -\Delta_{\mathbb R^n}\bigl(x_j a(|x|)\bigr)
    =x_j\left(-a''(|x|)-\frac{n+1}{|x|}a'(|x|)\right).
\]
Up to a constant normalization, $a\mapsto x_j a(|x|)$ is a unitary
map from the radial subspace of $L^2(\mathbb R^{n+2})$ onto the
corresponding degree-one angular subspace of $L^2(\mathbb R^n)$
(the odd subspace when $n=1$). The displayed identity intertwines
the self-adjoint Laplacians on these subspaces, so spectral calculus
gives the same identity for their fractional powers.
Differentiating the equation for $Q$ and using
$\partial_{x_j}Q=-x_j\psi(|x|)$ therefore yields
\[
    \bigl(( -\Delta_{\mathbb R^{n+2}})^s+1\bigr)\psi
    =(1+\alpha)Q(|x|)^\alpha\psi
    \quad\text{in }\mathbb R^{n+2}.
\]
Here $Q(|x|)$ denotes the original radial profile evaluated at the
radius in $\mathbb R^{n+2}$.
Thus $\psi$ is an $L^2$ eigenfunction with eigenvalue $-1$ for the
potential $-(1+\alpha)Q^\alpha$, which is bounded and vanishes at
infinity. By \cite[Lemma C.2]{FLS16},
\[
    \psi(r)\lesssim(1+r)^{-n-2s-2},\qquad
    \psi(r)=c_Q(n+2s)r^{-n-2s-2}+o(r^{-n-2s-2}).
\]
Here $c_Q>0$, since the asymptotic coefficient is a positive
multiple of $\int_{\mathbb R^{n+2}}Q^\alpha\psi\,dx>0$.
Finally, $Q'(r)=-r\psi(r)$ and
$Q(r)=\int_r^\infty t\psi(t)\,dt$ prove the assertions.
\end{proof}

\begin{lemma}\label{norm}
For every finite configuration with $m\ge2$, the interaction error
belongs to $H^1(\mathbb R^n)$ and satisfies
\begin{equation}\label{eq:interaction-norm-equivalence}
    C^{-1}\|f\|_{L^2}
    \le \|f\|_{H^{-s}}
    \le \|f\|_{L^2}
    \le \|f\|_{H^s}
    \le C\|f\|_{L^2},
\end{equation}
where $C=C(n,s,\alpha)$ is independent of both $m$ and the centers.
\end{lemma}
\begin{proof}
Since $0\le f\le \sigma^{1+\alpha}\le
\|\sigma\|_{L^\infty}^{\alpha}\sigma$ and
$\sigma\in L^2\cap L^\infty$, we have $f\in L^2$.
Differentiating $f$ and using $|\nabla Q|\le CQ$ gives
\[
    |\nabla f|
    \le(1+\alpha)\sum_{j=1}^m
        (\sigma^\alpha-Q_j^\alpha)|\nabla Q_j|
    \le C\sum_{j=1}^m(\sigma^\alpha-Q_j^\alpha)Q_j=Cf.
\]
Thus $f\in H^1$, with
\begin{equation}\label{eq:interaction-gradient-bound}
    \|\nabla f\|_{L^2}\le C\|f\|_{L^2}.
\end{equation}
For the energy norm used throughout the paper, Plancherel's theorem
and H\"older's inequality give
\[
    \|f\|_{H^s}^2
    =\|f\|_{L^2}^2+
      \int_{\mathbb R^n}|\xi|^{2s}|\widehat f(\xi)|^2\,d\xi
    \le\|f\|_{L^2}^2+
       \|f\|_{L^2}^{2(1-s)}\|\nabla f\|_{L^2}^{2s}.
\]
Consequently,
\begin{equation}\label{eq:interaction-Hs-bound}
    \|f\|_{H^s}
    \le C\|f\|_{L^2}.
\end{equation}
The inequalities $\|f\|_{H^{-s}}\le\|f\|_{L^2}\le\|f\|_{H^s}$
follow directly from the definitions. Finally, duality yields
\[
    \|f\|_{L^2}^2
    \le\|f\|_{H^{-s}}\|f\|_{H^s}
    \le C\|f\|_{H^{-s}}\|f\|_{L^2},
\]
which proves the remaining inequality.
\end{proof}

For $a,b\ge0$, write
\[
    D_\alpha(a,b):=(a+b)^{1+\alpha}-a^{1+\alpha}-b^{1+\alpha}.
\]
This scalar interaction will also be used below.

\begin{lemma}\label{decoupling}
For $\alpha>0$ and $m\ge2$,
\[
    0\le f\le
    \max\{1,(m-1)^{\alpha-1}\}
    \sum_{1\le i<j\le m}
    D_\alpha(Q_i,Q_j).
\]
\end{lemma}
\begin{proof}
We claim that, for $a_1,\ldots,a_N,b\ge0$,
\begin{equation}\label{eq:interaction-sum-bound}
    D_\alpha\left(\sum_{i=1}^N a_i,b\right)
    \le\max\{1,N^{\alpha-1}\}\sum_{i=1}^N D_\alpha(a_i,b).
\end{equation}
For fixed $y\ge0$, let $g_y(t)=(t+y)^\alpha-y^\alpha$.
If $0<\alpha\le1$, concavity and $g_y(0)=0$ imply
$g_y(a_i)\ge \dfrac{a_i}{\sum_j a_j}g_y(\sum_j a_j)$
when $\sum_j a_j>0$.
Summation gives
\begin{equation}\label{eq:concave-increment-bound}
    g_y\left(\sum_{i=1}^N a_i\right)\le\sum_{i=1}^N g_y(a_i);
\end{equation}
The case $\sum_j a_j=0$ is immediate. If $\alpha>1$, then $g_y$ is convex and,
for $\lambda\ge1$,
\[
    g_y(\lambda t)
    =\alpha\lambda\int_0^t(y+\lambda r)^{\alpha-1}\,dr
    \le\lambda^\alpha g_y(t).
\]
Jensen's inequality therefore gives
\begin{equation}\label{eq:convex-increment-bound}
    g_y\left(\sum_{i=1}^N a_i\right)
    \le\frac1N\sum_{i=1}^N g_y(Na_i)
    \le N^{\alpha-1}\sum_{i=1}^N g_y(a_i).
\end{equation}
Since $D_\alpha(a,b)=(1+\alpha)\int_0^b g_y(a)\,dy$,
integrating these inequalities proves \eqref{eq:interaction-sum-bound}.
The telescoping identity
\[
    f=\sum_{k=2}^mD_\alpha\left(\sum_{i=1}^{k-1}Q_i,Q_k\right)
\]
and \eqref{eq:interaction-sum-bound} give the assertion.
\end{proof}

\begin{lemma}\label{l2}
For $|y|>1$,
\begin{align*}
    &\int_{\{Q(x)>Q(x+y)\}}Q(x)^{2\alpha}Q(x+y)^2\,dx\\
    &\qquad\simeq
    \begin{cases}
        |y|^{-2(n+2s)},
        &\displaystyle \alpha>\alpha_*,\\[3pt]
        |y|^{-2(n+2s)}\log(2+|y|),
        &\displaystyle \alpha=\alpha_*,\\[3pt]
        |y|^{n-2(1+\alpha)(n+2s)},
        &\displaystyle 0<\alpha<\alpha_*.
    \end{cases}
\end{align*}
\end{lemma}
\begin{proof}
By rotation invariance, take $y=Re_1$, where $R=|y|$.
Strict radial monotonicity of $Q$ identifies the integration region as
$D_R=\{x_1>-R/2\}$. On $B_{R/2}\subset D_R$, we have
$R/2\le|x+Re_1|\le3R/2$. On $D_R\setminus B_{R/2}$,
\[
    |x|\le|x+Re_1|\le|x|+R\le3|x|.
\]
Moreover, $D_R\setminus B_{R/2}$ contains the half-space portion
$\{x_1\ge0,\ |x|>R/2\}$. The two-sided decay estimate for $Q$
therefore gives
\begin{align*}
    \int_{D_R}Q(x)^{2\alpha}Q(x+Re_1)^2\,dx
    &\simeq
    R^{-2(n+2s)}\int_0^{R/2}
       \frac{r^{n-1}}{(1+r)^{2\alpha(n+2s)}}\,dr\\
    &\quad+\int_{R/2}^\infty
       r^{n-1-2(1+\alpha)(n+2s)}\,dr.
\end{align*}
The second integral is comparable to
$R^{n-2(1+\alpha)(n+2s)}$. The first radial integral is comparable to
$1$, $\log(2+R)$, or $R^{n-2\alpha(n+2s)}$, according as
$2\alpha(n+2s)$ is greater than, equal to, or less than $n$.
These three cases prove the lemma.
\end{proof}

\begin{lemma}\label{eqf}
For $R\ge2$, the interaction error satisfies
\[
    \|f\|_{H^{-s}}\simeq
    \begin{cases}
        R^{-(n+2s)},
        &\displaystyle \alpha>\alpha_*,\\[3pt]
        R^{-(n+2s)}(\log R)^{\frac{1}{2}},
        &\displaystyle \alpha=\alpha_*,\\[3pt]
        R^{\frac{n}{2}-(1+\alpha)(n+2s)},
        &\displaystyle 0<\alpha<\alpha_*.
    \end{cases}
\]
\end{lemma}
\begin{proof}
Lemmas \ref{norm} and \ref{decoupling} imply
\[
    \|f\|_{H^{-s}}\simeq\|f\|_{L^2}
    \lesssim\sum_{i<j}\|D_\alpha(Q_i,Q_j)\|_{L^2}.
\]
For $a\ge b\ge0$, we have $D_\alpha(a,b)\simeq a^\alpha b$.
Indeed, $D_\alpha(a,b)/(a^\alpha b)$ is a positive continuous function
of $b/a\in(0,1]$, with limit $1+\alpha$ at zero.
Hence
\[
    \|D_\alpha(Q_i,Q_j)\|_{L^2}^2
    \lesssim
    \int_{\{Q_i>Q_j\}}Q_i^{2\alpha}Q_j^2\,dx+
    \int_{\{Q_j>Q_i\}}Q_j^{2\alpha}Q_i^2\,dx.
\]
The equality set is a hyperplane and has measure zero. The reflection
$x\mapsto-y_i-y_j-x$ interchanges the two profiles, so the two
integrals are equal. After translation and rotation, Lemma \ref{l2}
applies with $|y|=|y_i-y_j|$. Each resulting bound decreases for
$|y_i-y_j|\ge R\ge2$; summing over the finitely many pairs proves the upper bound.

For the lower bound, choose $i,j$ with $|y_i-y_j|=R$.
Positivity and superadditivity give $f\ge D_\alpha(Q_i,Q_j)$.
The scalar comparison above, Lemma~\ref{l2}, and reflection symmetry
give the matching lower bound for $\|D_\alpha(Q_i,Q_j)\|_{L^2}$.
Since $\|f\|_{H^{-s}}\simeq\|f\|_{L^2}$ by Lemma~\ref{norm},
the same lower bound holds for $f$.
\end{proof}

The norm of $f$ measures the total interaction. Its projections onto
translation modes detect a smaller scale, which will determine the
stability modulus. The gradient asymptotics in
Lemma \ref{lem:ground-state-gradient} identify their leading term.

\begin{lemma}\label{lem:fQ}
For fixed $m\ge2$ and all sufficiently large $R$,
\begin{equation}\label{eq:maximal-projection}
    \max_{1\le i\le m}
    \left|\int_{\mathbb R^n}f\nabla Q_i\,dx\right|
    \simeq R^{-n-2s-1}.
\end{equation}
The constants are independent of the centers.
\end{lemma}
\begin{proof}
We first prove the expansion
\begin{equation}\label{eq:projection-expansion}
    \int_{\mathbb R^n}f\nabla Q_i\,dx
    =-\left(\int Q^{1+\alpha}\,dx\right)
       \sum_{j\ne i}\nabla Q(y_j-y_i)
       +o(R^{-n-2s-1}),
\end{equation}
uniformly over all configurations with minimum separation $R$.
Lemma \ref{decoupling} and the two-profile estimate used above give
\[
    0\le f\lesssim\sum_{j<k}
       (Q_j^\alpha Q_k+Q_k^\alpha Q_j).
\]
Since $Q\in L^1$, while $Q^\alpha$ is bounded and tends to zero,
dominated convergence yields
\begin{equation}\label{eq:interaction-L1-small}
    \int_{\mathbb R^n}Q(x)^\alpha Q(x-y)\,dx\longrightarrow0
    \quad\text{as }|y|\to\infty,
    \qquad \|f\|_{L^1}=o(1).
\end{equation}
Indeed, after the change of variables $z=x-y$, the integrand is
$Q(z+y)^\alpha Q(z)$ and is dominated by
$\|Q\|_{L^\infty}^\alpha Q(z)$. The convergence is uniform for $|y|\ge R$
as $R\to\infty$.

Fix $i$ and translate by $-y_i$, so that the distinguished profile is
$Q(x)$. Write
\[
    b(x)=\sum_{j\ne i}Q(x+y_j-y_i),\qquad
    h(x)=b(x)^{1+\alpha}-\sum_{j\ne i}Q(x+y_j-y_i)^{1+\alpha}.
\]
Then $f(x-y_i)=D_\alpha(Q(x),b(x))+h(x)$.
Choose a fixed $\delta\in(0,1/4)$, depending only on the fixed
parameters, sufficiently small that $b\le Q$ on $B_{\delta R}$
for all large $R$. This follows from
$b\le C(m-1)((1-\delta)R)^{-n-2s}$ and
$Q\ge c(1+\delta R)^{-n-2s}$ there. On the same ball,
\begin{equation}\label{eq:distant-profile-bounds}
    b\lesssim R^{-n-2s},\qquad
    |\nabla b|\lesssim R^{-n-2s-1}.
\end{equation}
The part outside this ball is negligible: by
\eqref{eq:interaction-L1-small},
\begin{equation}\label{eq:projection-exterior}
    \left|\int_{|x|\ge\delta R}f(x-y_i)\nabla Q(x)\,dx\right|
    \lesssim R^{-n-2s-1}\|f\|_{L^1}
    =o(R^{-n-2s-1}).
\end{equation}

For the integral over $B_{\delta R}$, the chain rule applied to
$\int_0^{Q(x)}D_\alpha(t,b(x))\,dt$ and integration by parts give
\begin{align}\label{eq:primitive-identity}
    \int_{B_{\delta R}}D_\alpha(Q,b)\nabla Q\,dx
    &=\int_{\partial B_{\delta R}}
       \left(\int_0^QD_\alpha(t,b)\,dt\right)\frac{x}{\delta R}\,dS\nonumber\\
    &\quad-\int_{B_{\delta R}}
       \bigl[(Q+b)^{1+\alpha}-b^{1+\alpha}-(1+\alpha)Qb^\alpha\bigr]
       \nabla b\,dx.
\end{align}
Since $b\le Q$ on this ball,
\[
    \left|\int_0^QD_\alpha(t,b)\,dt\right|\lesssim Q^{1+\alpha}b,
    \qquad
    \bigl|(Q+b)^{1+\alpha}-b^{1+\alpha}-(1+\alpha)Qb^\alpha\bigr|
       \lesssim Q^{1+\alpha}.
\]
The boundary integral is therefore bounded by
\[
    CR^{n-1-(2+\alpha)(n+2s)}=o(R^{-n-2s-1}).
\]
On every fixed compact set, the coefficient of $\nabla b$ in
\eqref{eq:primitive-identity} tends to $Q^{1+\alpha}$.
Dominated convergence thus shows that its integral tends to
$\int Q^{1+\alpha}\,dx$.

To replace $\nabla b(x)$ by $\nabla b(0)$, fix $M>0$.
The gradient asymptotics give
\[
    \sup_{|x|\le M}|\nabla b(x)-\nabla b(0)|
       \le\left(o_R(1)+\frac{CM}{R}\right)R^{-n-2s-1}.
\]
Here the leading term is differentiated only where $|z|\ge R-M$,
and the asymptotic remainder is uniform there.
On $M<|x|<\delta R$, \eqref{eq:distant-profile-bounds} instead bounds
the difference by $CR^{-n-2s-1}$. Consequently,
\[
    R^{n+2s+1}\int_{B_{\delta R}}Q^{1+\alpha}
       |\nabla b(x)-\nabla b(0)|\,dx
    \lesssim o_R(1)+\frac{M}{R}
       +\int_{|x|>M}Q^{1+\alpha}\,dx.
\]
Letting first $R\to\infty$ and then $M\to\infty$ proves that the
replacement error is negligible, without a moment assumption.
Substitution into \eqref{eq:primitive-identity} gives
\begin{equation}\label{eq:principal-core-projection}
    \int_{B_{\delta R}}D_\alpha(Q,b)\nabla Q\,dx
    =-\left(\int Q^{1+\alpha}\,dx\right)\nabla b(0)
       +o(R^{-n-2s-1}).
\end{equation}

It remains to estimate the term involving $h$. On $B_{\delta R}$,
\[
    |h|\lesssim R^{-(1+\alpha)(n+2s)},\qquad
    |\nabla h|\lesssim R^{-(1+\alpha)(n+2s)-1}.
\]
An integration by parts, using $Q\in L^1$, gives
\begin{align*}
    \left|\int_{B_{\delta R}}h\nabla Q\,dx\right|
    &\le\int_{\partial B_{\delta R}}Q|h|\,dS
       +\int_{B_{\delta R}}Q|\nabla h|\,dx\\
    &\lesssim R^{n-1-(2+\alpha)(n+2s)}
       +R^{-(1+\alpha)(n+2s)-1}
     =o(R^{-n-2s-1}).
\end{align*}
Together with \eqref{eq:projection-exterior} and
\eqref{eq:principal-core-projection}, this proves
\eqref{eq:projection-expansion}.

The gradient asymptotics turn this expansion into
\begin{equation}\label{eq:projection-discrete-force}
    \int_{\mathbb R^n}f\nabla Q_i\,dx
    =c_Q(n+2s)\left(\int Q^{1+\alpha}\,dx\right)
       \sum_{j\ne i}\frac{y_j-y_i}{|y_j-y_i|^{n+2s+2}}
       +o(R^{-n-2s-1}).
\end{equation}
Since every distance is at least $R$, this gives the upper bound
in \eqref{eq:maximal-projection}.

For the lower bound, choose a closest pair $y_a,y_b$, set
$e=(y_b-y_a)/R$, and let
\[
    S=\left\{i:y_i\cdot e<\frac{(y_a+y_b)\cdot e}{2}\right\}.
\]
Then $a\in S$ and $b\notin S$. The terms with both indices in $S$
cancel pairwise, leaving
\[
    \sum_{i\in S}\sum_{j\ne i}
       \frac{(y_j-y_i)\cdot e}{|y_j-y_i|^{n+2s+2}}
    =\sum_{\substack{i\in S\\j\notin S}}
       \frac{(y_j-y_i)\cdot e}{|y_j-y_i|^{n+2s+2}}
    \ge\frac{(y_b-y_a)\cdot e}{|y_b-y_a|^{n+2s+2}}
    =R^{-n-2s-1}.
\]
Every summand is positive, which justifies retaining only the closest
pair. Summing \eqref{eq:projection-discrete-force} over $i\in S$ and
using the uniform remainder, we obtain, for all sufficiently large $R$,
\[
    \sum_{i\in S}\left(\int_{\mathbb R^n}f\nabla Q_i\,dx\right)\cdot e
    \ge\frac{c_Q(n+2s)}{2}\left(\int Q^{1+\alpha}\,dx\right)
       R^{-n-2s-1}.
\]
At least one of these at most $m$ projections therefore has magnitude
at least $cR^{-n-2s-1}$. This proves the lower bound and
completes the proof.
\end{proof}

\section{The proof of Theorem \ref{thm-quantitative}}\label{sec3}

We use the configuration notation in \eqref{eq:configuration}, with $R$
sufficiently large.
We equip \(H^s(\mathbb R^n)\) with the energy inner product
\[
    \langle v,w\rangle_{H^s}
    :=
    \int_{\mathbb R^n} vw \,dx
    +
    \int_{\mathbb R^n}  (-\Delta)^{\frac{s}{2}}v\,(-\Delta)^{\frac{s}{2}}w\,dx.
\]
Choose the parameters to minimize
\[
    (y_1,\ldots,y_m)\longmapsto
    \left\|u-\sum_{i=1}^m Q(\cdot+y_i)\right\|_{H^s}^2,
\]
and write $u=\sigma+\rho$.
Near a sufficiently separated configuration, this minimum is attained:
any minimizing sequence gives sums close to the initial sum, so positivity
and separation match its $m$ peaks to the initial peaks. After a permutation,
the parameters remain in bounded neighborhoods of the initial parameters;
passing to a subsequence and using continuity gives a minimizer.
For small $\varepsilon_0$ these neighborhoods can be chosen uniformly small,
so the chosen configuration still satisfies $R\ge R_0$.
Differentiating with respect to each parameter coordinate gives
\[
    \langle\rho,\partial_{x_j}Q_i\rangle_{H^s}=0,
    \qquad
    1\le i\le m,\quad 1\le j\le n.
\]
Since
\[
    L\partial_{x_j}Q_i
    =(1+\alpha)Q_i^\alpha\partial_{x_j}Q_i,
\]
these conditions can equivalently be written as
\begin{equation}\label{eq:modulation-orthogonality}
    \langle\rho,\partial_{x_j}Q_i\rangle_{H^s}
    =
    (1+\alpha)
    \int_{\mathbb R^n}
    \rho Q_i^\alpha\partial_{x_j}Q_i\,dx
    =0.
\end{equation}
Let
\[
\mathscr F := \vspan\left\{  \partial_{x_j } Q_i, ~1 \leq i \leq m,~1\le j\le n \right\}.
\]
Let \( P_{\mathscr F^\perp} \) (respectively \( P_{\mathscr F} \)) denote the orthogonal projection from \( H^s \) onto \( \mathscr F^\perp \) (respectively \( \mathscr F \)). Thus, $\rho\in\mathscr F^\perp$.

Define the compact operator \( K \) by
\[
Kv=(1+\alpha)L^{-1}(\sigma^\alpha v),\qquad K:H^s\to H^s.
\]

We have the following result:
\begin{lemma}\label{reduction}
If $R$ is sufficiently large, then the problem
\[
  (\mathrm{id}-P_{\mathscr F^\perp}K)v=\varphi,\qquad v\in\mathscr F^\perp,
\]
is uniquely solvable for every $\varphi\in\mathscr F^\perp$, and
\begin{equation}
\|v\|_{H^s}\simeq \|\varphi\|_{H^s}.
\end{equation}
\end{lemma}
\begin{proof}
  We first prove the a priori estimate
  \begin{equation}\label{prio}\|v\|_{H^s}\leq C\|\varphi\|_{H^s}.\end{equation}
Otherwise, take configurations $y_{i,l}$ whose minimum separation tends
to infinity, and write $Q_{i,l}(x)=Q(x+y_{i,l})$.
Subscripts $l$ on $\sigma,\mathscr F,K$ and their associated projections
indicate this configuration. There are
$v_l,\varphi_l\in\mathscr F_l^\perp$ such that
\[
  (\mathrm{id}-P_{\mathscr F_l^\perp}K_l)v_l=\varphi_l,
  \qquad\|v_l\|_{H^s}=1,\qquad\|\varphi_l\|_{H^s}\to0.
\]
Let $b_{ij}^l$ be the $H^s$ orthonormal basis obtained by applying
Gram--Schmidt to the $mn$ functions $\partial_{x_j}Q_{i,l}$.
The projected equation gives
\begin{equation}
  Lv_l-(1+\alpha)\sigma_l^\alpha v_l
  =L\varphi_l-(1+\alpha)\sum_{i=1}^m\sum_{j=1}^n
       \left(\int\sigma_l^\alpha v_l b_{ij}^l\,dx\right)Lb_{ij}^l.
\end{equation}
Separation and radial symmetry imply
\[
  \left\|b_{ij}^l-
       \frac{\partial_{x_j}Q_{i,l}}{\|\partial_{x_j}Q\|_{H^s}}
  \right\|_{H^s}\to0.
\]
Using $v_l\in\mathscr F_l^\perp$, boundedness of $\sigma_l$,
and decay, we obtain
\[
  \left|\int\sigma_l^\alpha v_l b_{ij}^l\,dx\right|
  \lesssim\|(\sigma_l^\alpha-Q_{i,l}^\alpha)
                 \partial_{x_j}Q_{i,l}\|_{L^2}+o_l(1)=o_l(1).
\]
After passing to a subsequence,
$v_l(\cdot-y_{i,l})\rightharpoonup v_i$ in $H^s$, where
\[
L v_i - (\alpha+1)Q^{\alpha} v_i = 0, \quad \int_{\mathbb R^n } Q^{\alpha} \partial_{x_j } Q v_i dx = 0.
\]
Therefore $v_i=0$, and $v_l(\cdot-y_{i,l})\to0$ in $L^2_{\mathrm{loc}}$.

For a fixed \( M \) large enough such that $Q^{\alpha}(x)\leq \frac{1}{100(\alpha+1)m^{\alpha+1}}$ in $|x|>M$, test the equation with \( v_l \), we obtain
  \begin{equation}
  \begin{aligned}
\|v_l\|_{H^s}^2&= \int_{\mathbb R^n} (\alpha+1)\sigma_l^{\alpha} v_l^2 dx +o_l(1)\\&
\leq C\int_{\cup_{i=1}^{m} \{|x+y_{i,l}|\le M\} } v_l^2 dx\\&
\quad{}+\int_{\mathbb R^n \setminus \cup_{i=1}^{m} \{|x+y_{i,l}|\le M\}} (\alpha +1) \sigma_l^{\alpha} v_l^2 dx +o_l(1)
\\&
\leq o_l(1)+\frac{1}{100}\|v_l\|_{L^2}^2\\&
\leq \frac{1}{2}.
\end{aligned}
\end{equation}
This leads to a contradiction with $\|v_l\|_{H^s}=1.$  This establishes the uniform a priori estimate. For each fixed configuration, the operator
\[
    P_{\mathscr F^\perp} K\big|_{\mathscr F^\perp}:\mathscr F^\perp\to \mathscr F^\perp
\]
is compact. The a priori estimate implies that
\(\mathrm{id}-P_{\mathscr F^\perp}K\) is injective on \(\mathscr F^\perp\).
The Fredholm's alternative therefore gives solvability.

Finally,
\[
    \|Kv\|_{H^s}
    \le(1+\alpha)\|\sigma\|_{L^\infty}^{\alpha}\|v\|_{H^s}
    \le(1+\alpha)m^\alpha\|Q\|_{L^\infty}^{\alpha}\|v\|_{H^s},
\]
so
\[
    \|\varphi\|_{H^s}
    \lesssim \|v\|_{H^s}.
\]
Together with the a priori estimate, this establishes
the two-sided bound.
\end{proof}

\begin{lemma}\label{lem:nonlinear-estimates}
For $\|v\|_{H^s},\|w\|_{H^s}\le1$, uniformly over separated configurations,
\begin{align}
 \|N_\sigma(v)\|_{H^{-s}}
 &\lesssim\|v\|_{H^s}^{1+\min\{\alpha,1\}},\label{eq:remainder-bound}\\
 \|N_\sigma(v)-N_\sigma(w)\|_{H^{-s}}
 &\lesssim
 \bigl(\|v\|_{H^s}^{\min\{\alpha,1\}}
       +\|w\|_{H^s}^{\min\{\alpha,1\}}\bigr)
       \|v-w\|_{H^s}.\label{eq:remainder-difference}
\end{align}
For every translation mode,
\begin{equation}\label{3.36}
 |\langle N_\sigma(v),\nabla Q_i\rangle|
 \lesssim\|v\|_{H^s}^2,
\end{equation}
and
\begin{equation}\label{4.111}
 \left|\int(\sigma^\alpha-Q_i^\alpha)v\nabla Q_i\,dx\right|
 \lesssim\|v\|_{H^s}^{\alpha+2}
       +R^{-n-2s}\|v\|_{H^s}^{\min\{\alpha,1\}}.
\end{equation}
If $n\in\{1,2\}$ and $0<\alpha<1$, then
\begin{equation}\label{eq:linear-defect}
 \|(\sigma^\alpha-Q_i^\alpha)\nabla Q_i\|_{L^2}
 \lesssim R^{-n-2s}.
\end{equation}
\end{lemma}

\begin{proof}
The scalar Taylor estimates give
\[
 |N_\sigma(v)-N_\sigma(w)|\lesssim
 \begin{cases}
 (|v|^\alpha+|w|^\alpha)|v-w|,&0<\alpha\le1,\\
 [\sigma^{\alpha-1}(|v|+|w|)+|v|^\alpha+|w|^\alpha]|v-w|,
       &\alpha>1.
 \end{cases}
\]
Sobolev embedding, its dual, and H\"older's inequality prove
\eqref{eq:remainder-difference}, since
$\|\sigma\|_{L^{\alpha+2}}$ is uniformly bounded.
Taking $w=0$ gives \eqref{eq:remainder-bound}.

For \eqref{3.36}, use $|\nabla Q_i|\lesssim Q_i\le\sigma\lesssim1$
and split into $|v|\le\sigma$ and $|v|>\sigma$:
\[
 |\langle N_\sigma(v),\nabla Q_i\rangle|
 \lesssim\int_{\{|v|\le\sigma\}}\sigma^\alpha|v|^2\,dx
       +\int_{\{|v|>\sigma\}}|v|^{\alpha+2}\,dx
 \lesssim\|v\|_{H^s}^2.
\]
For the linear interaction, boundedness of the profiles and
$|\nabla Q_i|\lesssim Q_i$ give
\[
 |(\sigma^\alpha-Q_i^\alpha)v\nabla Q_i|
 \lesssim |v|^{\alpha+2}
       +Q_i(\sigma-Q_i)|v|^{\min\{\alpha,1\}}.
\]
Indeed, split according to $Q_i>\sigma-Q_i$ and its complement,
and then compare $|v|$ with the larger of these two quantities.
Decay gives
$\|Q_i(\sigma-Q_i)\|_{L^p}\lesssim R^{-n-2s}$ for $1\le p\le2$.
H\"older's inequality with $p=2/(2-\min\{\alpha,1\})$ proves
\eqref{4.111}.

Finally, on $B_{R/2}(-y_i)$, concavity gives
\[
 |(\sigma^\alpha-Q_i^\alpha)\nabla Q_i|
 \lesssim R^{-n-2s}(1+|x+y_i|)^{-\alpha(n+2s)-1}.
\]
This has $L^2$ norm at most $CR^{-n-2s}$ when $n\le2$.
On the complement, boundedness of $\sigma$ and gradient decay give
\[
 \|(\sigma^\alpha-Q_i^\alpha)\nabla Q_i
       \|_{L^2(\mathbb R^n\setminus B_{R/2}(-y_i))}
 \lesssim R^{n/2-n-2s-1}\lesssim R^{-n-2s}.
\]
Together these prove \eqref{eq:linear-defect}.
\end{proof}

Define
\[
    W(x):=\sum_{i=1}^m(1+|x+y_i|)^{-\alpha(n+2s)},
\]
and
\[
X:=
    \left\{
        v\in C(\mathbb R^n):
        \|v\|_X:=\sup_{x\in\mathbb R^n}
        \frac{|v(x)|}{W(x)}<\infty
    \right\}.
\]
\begin{lemma}\label{lem:weighted-linear-small-alpha}
Assume \eqref{eq:crit_exp} and $0<\alpha<\alpha_*$.
For sufficiently large \(R\), the solution
\[v=\left[
        (\mathrm{id}-P_{\mathscr F^\perp}K)
        \big|_{\mathscr F^\perp}
    \right]^{-1}
    P_{\mathscr F^\perp} L^{-1} g\]
belongs to \(H^s\cap X\cap\mathscr F^\perp\) and satisfies
\[
    \|v\|_X\lesssim\|g\|_X,
    \qquad
    \|v\|_{H^s}\lesssim\|g\|_{H^{-s}},
    \qquad g\in L^2(\mathbb R^n)\cap X.
\]
The implicit constants are independent of the centers.
\end{lemma}

\begin{proof}
Existence, uniqueness, and the \(H^s\) estimate follow
directly from Lemma~\ref{reduction}. It remains to prove
that \(v\in X\) and establish the uniform weighted bound.

Let $G$ denote the convolution kernel of
$L^{-1}$, so that
\[
    \widehat G(\xi)=\frac{1}{1+|\xi|^{2s}}.
\]
The kernel $G$ is positive and radial, with
\[
    G\in L^1(\mathbb R^n),
    \qquad
    \int_{\mathbb R^n}G(x)\,dx=1.
\]
Its behaviour near the origin is given by
\begin{equation}\label{eq:resolvent-kernel-origin}
    G(x)\simeq
    \begin{cases}
        |x|^{-(n-2s)}, & n>2s,\\[2mm]
        \log\displaystyle\frac{2}{|x|}, & n=2s,\\[2mm]
        1, & n<2s,
    \end{cases}
    \qquad 0<|x|\le1,
\end{equation}
whereas
\begin{equation}\label{eq:resolvent-kernel-infinity}
    G(x)\simeq |x|^{-(n+2s)},
    \qquad |x|\ge1.
\end{equation}
Indeed, the fractional heat-kernel representation gives
\[
    G(x)\simeq\int_0^\infty e^{-t}
    \min\left\{
        t^{-\frac{n}{2s}},\frac{t}{|x|^{n+2s}}
    \right\}\,dt.
\]
Splitting the integral at $t=|x|^{2s}$ gives
\eqref{eq:resolvent-kernel-origin} and
\eqref{eq:resolvent-kernel-infinity}.

Consequently,
\begin{equation}\label{eq:resolvent-kernel-integrability}
    \begin{cases}
        G\in L^r(\mathbb R^n)
        \quad\text{for }1\le r<\dfrac{n}{n-2s},
        & n>2s,\\[2mm]
        G\in L^r(\mathbb R^n)
        \quad\text{for }1\le r<\infty,
        & n=2s,\\[2mm]
        G\in L^r(\mathbb R^n)
        \quad\text{for }1\le r\le\infty,
        & n<2s.
    \end{cases}
\end{equation}
The integrability of $G$ and its decay at infinity also imply
\begin{equation}\label{eq:weighted-convolution-refinement}
    G*(1+|\cdot|)^{-d}(x)
    \le C(1+|x|)^{-d},
    \qquad 0<d<n.
\end{equation}
The constant in this estimate may depend on $d$.
In particular, by translation invariance,
\[
    G*W\lesssim W,
\]
with a constant independent of the centers.

The projected equation is equivalent to
\begin{equation}\label{eq:weighted-integral-equation}
    v
    =
    G*\bigl((1+\alpha)\sigma^\alpha v+g\bigr)
    -\sum_{i,j}c_{i,j}\partial_{x_j}Q_i
\end{equation}
for suitable coefficients \(c_{i,j}\).

We first show that the $H^s$ solution belongs to $X$.
Since $g\in L^2\cap X$ and $W\le C$, we have
$g\in L^2\cap L^\infty$. Hence
\[
    G*g-\sum_{i,j}c_{i,j}\partial_{x_j}Q_i
    \in L^2\cap L^\infty.
\]
Moreover, $\sigma^\alpha$ is bounded.

If $n\le2s$, then $G\in L^2$, and Young's inequality gives
\[
    \|G*(\sigma^\alpha v)\|_{L^\infty}
    \le
    \|G\|_{L^2}\|\sigma^\alpha v\|_{L^2}<\infty.
\]
Thus $v\in L^\infty$.

If $n>2s$, start from $v\in L^2$.
Whenever $v\in L^p$ with $p\le n/(2s)$, Young's inequality and
$G\in L^{n/(n-s)}$ yield $v\in L^{np/(n-sp)}$.
After finitely many steps, $v\in L^p$ for some
$p>\frac{n}{2s}$. Since $G\in L^{p/(p-1)}$, another application
of Young's inequality yields $v\in L^\infty$.

Using
\[
    \sigma^\alpha\lesssim W,
    \qquad
    |\partial_{x_j}Q_i(x)|
    \lesssim(1+|x+y_i|)^{-n-2s-1}
    \lesssim W(x),
\]
we deduce from \eqref{eq:weighted-integral-equation} that
\[
    |v(x)|
    \lesssim
    \bigl(\|v\|_{L^\infty}+\|g\|_X\bigr)(G*W)(x)
    +\sum_{i,j}|c_{i,j}|\,|\partial_{x_j}Q_i(x)|
    \le CW(x).
\]
At this stage the constant may depend on the solution.
Finally, convolution with $G\in L^1$ maps bounded functions
to continuous functions, by continuity of translations in
$L^1$. Thus, $v$ has a continuous representative and belongs
to $X$. This argument establishes membership only; we next
prove that the weighted bound is uniform.

Suppose otherwise. Use the configuration subscripts $l$ as in the proof
of Lemma~\ref{reduction}, including the corresponding weights $W_l$
and spaces $X_l$. There exist configurations with minimum separation
tending to infinity and solutions
\[v_l=\left[
        (\mathrm{id}-P_{\mathscr F_l^\perp}K_l )
        \big|_{\mathscr F_l^\perp}
    \right]^{-1}
    P_{\mathscr F_l^\perp} L^{-1} g_l\]
such that
\[
    \|v_l\|_{X_l}=1,\qquad
    \|g_l\|_{X_l}\to0.
\]
As in the proof of Lemma~\ref{reduction}, write
\[
    Lv_l-(1+\alpha)\sigma_l^\alpha v_l
    =
    g_l-\sum_{i,j}c_{i,j,l}L\partial_{x_j}Q_{i,l}.
\]
Testing against \(\partial_{x_b}Q_{a,l}\) and using
\(v_l\in\mathscr F_l^\perp\) gives
\begin{align*}
    \sum_{i,j}c_{i,j,l}
    \langle \partial_{x_j}Q_{i,l},\partial_{x_b}Q_{a,l}\rangle_{H^s}
    &=
    \int_{\mathbb R^n}g_l \partial_{x_b}Q_{a,l}\,dx\\
    &\quad+
    (1+\alpha)\int_{\mathbb R^n}
    (\sigma_l^\alpha-Q_{a,l}^\alpha)
    v_l \partial_{x_b}Q_{a,l}\,dx.
\end{align*}
Both terms tend to zero by separation and decay,
since \(|v_l|\le W_l\le C\).
The uniform invertibility of the Gram matrices,
as used in Lemma~\ref{reduction}, implies
\[
    c_{i,j,l}\to0.
\]

The integral equation gives local uniform compactness.
After translating to each center and passing to a
subsequence,
\[
    v_l(\cdot-y_{i,l})\to v_i
    \qquad\text{locally uniformly},
\]
where
\[
    Lv_i-(1+\alpha)Q^\alpha v_i=0,
    \qquad
    \int_{\mathbb R^n}
    v_iQ^\alpha\partial_{x_j}Q\,dx=0,
\]
and
\[
    |v_i(x)|\lesssim(1+|x|)^{-\alpha(n+2s)}.
\]
Iterating
\[
    v_i=(1+\alpha)G*(Q^\alpha v_i)
\]
and using \eqref{eq:weighted-convolution-refinement}
increases the decay exponent by \(\alpha(n+2s)\) at each step.
Since \(0<\alpha(n+2s)<\frac{n}{2}\), after finitely many steps the
exponent lies in \((\frac{n}{2},n)\).
Hence \(v_i\in L^2\), and the equation gives
\(v_i\in H^s\).
The nondegeneracy argument in Lemma~\ref{reduction}
therefore yields \(v_i=0\).

Finally, fix $M>1$ and split the integral equation over
$\bigcup_iB_M(-y_{i,l})$ and its complement.
The local uniform convergence to zero makes the first contribution
vanish as $l\to\infty$. On the complement,
$\sigma_l^\alpha\lesssim M^{-\alpha(n+2s)}$; using
$G*W_l\lesssim W_l$, $\|g_l\|_{X_l}\to0$, and $c_{i,j,l}\to0$ gives
\[
    1\le CM^{-\alpha(n+2s)}+o_l(1)
    \qquad\text{for each fixed }M.
\]
Letting $l\to\infty$ and then choosing $M$ sufficiently large gives
a contradiction.
This proves the uniform weighted estimate.
\end{proof}

\begin{lemma}\label{lem:corrected-configuration}
Assume \eqref{eq:crit_exp}. For sufficiently large $R$, there is a
unique $w_\sigma\in\mathscr F^\perp$ with sufficiently small $H^s$ norm
such that
\begin{equation}\label{eq:projected-correction}
 P_{\mathscr F^\perp}L^{-1}\mathcal E(U_\sigma)=0,
 \qquad U_\sigma=\sigma+w_\sigma.
\end{equation}
Moreover,
\begin{equation}\label{eq:correction-norm}
 \|w_\sigma\|_{H^s}
 \simeq\|P_{\mathscr F^\perp}L^{-1}f\|_{H^s}
 \lesssim\|f\|_{H^{-s}}.
\end{equation}
If $0<\alpha<\alpha_*$, then $w_\sigma\in X$ and
\begin{equation}\label{eq:weighted-correction-bounds}
 \|w_\sigma\|_{H^s}\lesssim R^{n/2-(1+\alpha)(n+2s)},
 \qquad |w_\sigma(x)|\lesssim R^{-n-2s}W(x).
\end{equation}
The constants and the smallness threshold are independent of the centers.
\end{lemma}

\begin{proof}
The projected equation is equivalent to $w=T(w)$, where
\[
 T(w):=
 \left[(\mathrm{id}-P_{\mathscr F^\perp}K)
          \big|_{\mathscr F^\perp}\right]^{-1}
 P_{\mathscr F^\perp}L^{-1}\bigl(f+N_\sigma(w)\bigr).
\]
Lemma~\ref{reduction} and \eqref{eq:remainder-bound} give
\[
 \|T(w)\|_{H^s}
 \lesssim\|f\|_{H^{-s}}+\|w\|_{H^s}^{1+\min\{\alpha,1\}}.
\]
Since $\|f\|_{H^{-s}}\to0$ by Lemma~\ref{eqf},
this map preserves the ball
$\|w\|_{H^s}\le C_*\|f\|_{H^{-s}}$ for fixed sufficiently large $C_*$
and large $R$. Equation~\eqref{eq:remainder-difference} makes it a
contraction there. The same difference estimate shows uniqueness
among all solutions with sufficiently small $H^s$ norm.

The fixed point equation gives
\[
 (\mathrm{id}-P_{\mathscr F^\perp}K)w_\sigma
 =P_{\mathscr F^\perp}L^{-1}
       \bigl(f+N_\sigma(w_\sigma)\bigr).
\]
Uniform boundedness and invertibility of the operator, together with
\eqref{eq:remainder-bound}, imply
\[
 \|w_\sigma\|_{H^s}\simeq
 \|P_{\mathscr F^\perp}L^{-1}f\|_{H^s},
\]
after absorbing the higher-order remainder. This proves
\eqref{eq:correction-norm}.

For $0<\alpha<\alpha_*$, restrict the same map to
\begin{equation}\label{eq:correction-ball}
 \|w\|_{H^s}\le C_*R^{n/2-(1+\alpha)(n+2s)},
 \qquad \|w\|_X\le C_*R^{-n-2s}.
\end{equation}
This closed subset of $\mathscr F^\perp\cap X$ is complete for the metric
induced by
\[
 R^{(1+\alpha)(n+2s)-n/2}\|w\|_{H^s}
       +R^{n+2s}\|w\|_X.
\]
Lemma~\ref{decoupling} and the two-profile bounds give
$|f|\lesssim R^{-n-2s}W$: in each pair, the smaller profile is at most
$CR^{-n-2s}$. Since $W\le C$ and $\alpha<1$, the scalar estimates used
in Lemma~\ref{lem:nonlinear-estimates} also yield
\begin{align*}
 \|N_\sigma(w)\|_X&\lesssim\|w\|_X^{1+\alpha},\\
 \|N_\sigma(w_1)-N_\sigma(w_2)\|_X
 &\lesssim(\|w_1\|_X^\alpha+\|w_2\|_X^\alpha)
             \|w_1-w_2\|_X.
\end{align*}
Here $f+N_\sigma(w)\in L^2\cap X$, so
Lemma~\ref{lem:weighted-linear-small-alpha} gives
\begin{align*}
 \|T(w)\|_X&\le CR^{-n-2s}+C\|w\|_X^{1+\alpha},\\
 \|T(w_1)-T(w_2)\|_X
 &\le CC_*^\alpha R^{-\alpha(n+2s)}\|w_1-w_2\|_X.
\end{align*}
Together with the $H^s$ estimates already proved and
Lemma~\ref{eqf}, these show that $T$ preserves \eqref{eq:correction-ball}
and is a contraction for the displayed metric, first choosing $C_*$
large and then $R$ large. Its fixed point equals $w_\sigma$ by the
$H^s$ uniqueness. This proves \eqref{eq:weighted-correction-bounds}.
\end{proof}

\begin{lemma}\label{lem:corrected-force}
Assume \eqref{eq:crit_exp}, and let $U_\sigma$ be the corrected
configuration from Lemma~\ref{lem:corrected-configuration}.
For sufficiently large $R$, uniformly in $i$,
\begin{equation}\label{eq:corrected-projection-expansion}
 \langle\mathcal E(U_\sigma),\nabla Q_i\rangle
 =-\int f\nabla Q_i\,dx+o(R^{-n-2s-1}),
\end{equation}
and
\begin{equation}\label{eq:corrected-residual-scale}
 \Gamma(U_\sigma)\simeq R^{-n-2s-1}.
\end{equation}
\end{lemma}

\begin{proof}
Since \(w_\sigma\in \mathscr F^\perp\),
\begin{align*}
    \langle\mathcal E(U_\sigma),\nabla Q_i\rangle
    &=
    -\int_{\mathbb R^n}f\nabla Q_i\,dx-(1+\alpha)\int_{\mathbb R^n}
    (\sigma^\alpha-Q_i^\alpha)w_\sigma\nabla Q_i\,dx\\
    &\quad
    -\int_{\mathbb R^n}N_\sigma(w_\sigma)\nabla Q_i\,dx.
\end{align*}

First suppose $\alpha\ge\alpha_*$. Equations~\eqref{3.36} and
\eqref{4.111}, together with \eqref{eq:correction-norm} and
Lemma~\ref{eqf}, bound the two error terms by
\[
 C\|f\|_{H^{-s}}^2
       +CR^{-n-2s}\|f\|_{H^{-s}}^{\min\{\alpha,1\}}.
\]
For $n\le2$ and $0<\alpha<1$, use \eqref{eq:linear-defect}
instead, replacing the last exponent by $1$.
In the present range,
$\|f\|_{H^{-s}}\lesssim R^{-n-2s}(1+\log R)^{1/2}$.
Since $n+2s>1$, its square and its product with $R^{-n-2s}$
are $o(R^{-n-2s-1})$.
In the remaining case $n\ge3$ and $\alpha_*\le\alpha<1$,
$\alpha(n+2s)\ge n/2>1$, so
\[
 R^{-n-2s}\|f\|_{H^{-s}}^\alpha
 \lesssim R^{-(1+\alpha)(n+2s)}(1+\log R)^{\alpha/2}
 =o(R^{-n-2s-1}).
\]
This proves \eqref{eq:corrected-projection-expansion} for
$\alpha\ge\alpha_*$, including the endpoint.

It remains to treat $0<\alpha<\alpha_*$. Fix $i$.
The global Taylor bound and $\sigma\ge Q_i$ give
\[
 |N_\sigma(w_\sigma)|\,|\nabla Q_i|
 \lesssim(1+|x+y_i|)^{-\alpha(n+2s)-1}|w_\sigma|^2.
\]
On $B_{R/2}(-y_i)$, concavity and
\eqref{eq:weighted-correction-bounds} give
\begin{align*}
 (\sigma^\alpha-Q_i^\alpha)|\nabla Q_i|
 &\lesssim R^{-n-2s}(1+|x+y_i|)^{-\alpha(n+2s)-1},\\
 |w_\sigma(x)|
 &\lesssim R^{-n-2s}(1+|x+y_i|)^{-\alpha(n+2s)}.
\end{align*}
Both core error terms are therefore bounded by the same radial integral:
\begin{align*}
 &\int_{B_{R/2}(-y_i)}
 \bigl((\sigma^\alpha-Q_i^\alpha)|w_\sigma|\,|\nabla Q_i|
       +|N_\sigma(w_\sigma)|\,|\nabla Q_i|\bigr)\,dx\\
 &\quad\lesssim R^{-2(n+2s)}
       \int_0^{R/2}r^{n-1}(1+r)^{-2\alpha(n+2s)-1}\,dr\\
 &\quad\lesssim
 \begin{cases}
 R^{n-2(1+\alpha)(n+2s)-1},&n>2\alpha(n+2s)+1,\\
 R^{-2(n+2s)}\log R,&n=2\alpha(n+2s)+1,\\
 R^{-2(n+2s)},&n<2\alpha(n+2s)+1
 \end{cases}\\
 &\quad=o(R^{-n-2s-1}).
\end{align*}
For the nonlinear term, the extra factor
$(1+|x+y_i|)^{-\alpha(n+2s)}\le1$ has simply been discarded.

On the complement, use
\[
 \sigma^\alpha-Q_i^\alpha
 \le(\sigma-Q_i)^\alpha
 \lesssim\sum_{j\ne i}(1+|x+y_j|)^{-\alpha(n+2s)}.
\]
Splitting into balls around the other centers and their complement gives
\[
 \|(\sigma^\alpha-Q_i^\alpha)\nabla Q_i
       \|_{L^2(\mathbb R^n\setminus B_{R/2}(-y_i))}
 \lesssim R^{n/2-(1+\alpha)(n+2s)-1}.
\]
Together with $\|w_\sigma\|_{L^2}\lesssim
R^{n/2-(1+\alpha)(n+2s)}$, this bounds the exterior linear error by
$CR^{n-2(1+\alpha)(n+2s)-1}$.
The global nonlinear bound above gives
\begin{align*}
 \int_{\mathbb R^n\setminus B_{R/2}(-y_i)}
       |N_\sigma(w_\sigma)|\,|\nabla Q_i|\,dx
 &\lesssim R^{-\alpha(n+2s)-1}\|w_\sigma\|_{L^2}^2\\
 &\lesssim R^{n-(2+3\alpha)(n+2s)-1}.
\end{align*}
Both exterior bounds are $o(R^{-n-2s-1})$.
This completes the proof of \eqref{eq:corrected-projection-expansion}.

Finally, \eqref{eq:projected-correction} implies
$L^{-1}\mathcal E(U_\sigma)\in\mathscr F$.
Uniform bounds for the Gram matrix of the translation modes give
\begin{equation}\label{eq:residual-projection-norm}
 \Gamma(U_\sigma)
 \simeq\left(\sum_{i=1}^m\sum_{j=1}^n
 |\langle\mathcal E(U_\sigma),\partial_{x_j}Q_i\rangle|^2
 \right)^{1/2}.
\end{equation}
Equation~\eqref{eq:corrected-projection-expansion} and
Lemma~\ref{lem:fQ} now prove \eqref{eq:corrected-residual-scale}.
\end{proof}

\begin{proof}[Proof of Theorem~\ref{thm-quantitative}]
Choose the centers by the minimization argument at the start of this
section. Then $\rho=u-\sigma\in\mathscr F^\perp$,
$\|\rho\|_{H^s}\le\varepsilon_0$, and $R\ge R_0$.
By Lemma~\ref{eqf} and local Lipschitz continuity of $\mathcal E$,
$\Gamma(u)\lesssim\varepsilon_0+\|f\|_{H^{-s}}$ is small.

Let $w_\sigma$ be given by Lemma~\ref{lem:corrected-configuration}.
Subtracting the projected equations for $\rho$ and $w_\sigma$ gives
\[
 (\mathrm{id}-P_{\mathscr F^\perp}K)(\rho-w_\sigma)
 =P_{\mathscr F^\perp}L^{-1}
       \bigl(\mathcal E(u)+N_\sigma(\rho)-N_\sigma(w_\sigma)\bigr).
\]
Lemma~\ref{reduction} and \eqref{eq:remainder-difference} yield
\begin{align*}
 \|\rho-w_\sigma\|_{H^s}
 &\lesssim\Gamma(u)\\
 &\quad+
 \bigl(\|\rho\|_{H^s}^{\min\{\alpha,1\}}
       +\|w_\sigma\|_{H^s}^{\min\{\alpha,1\}}\bigr)
       \|\rho-w_\sigma\|_{H^s}.
\end{align*}
Choose $\varepsilon_0$ small and $R_0$ large. Absorption gives
\begin{equation}\label{eq:comparison-corrected-profile}
 \|u-U_\sigma\|_{H^s}=\|\rho-w_\sigma\|_{H^s}\lesssim\Gamma(u).
\end{equation}
The relevant $H^s$ norms are uniformly bounded, so local Lipschitz
continuity of $\mathcal E$ implies
$\Gamma(U_\sigma)\lesssim\Gamma(u)$.
Lemma~\ref{lem:corrected-force} therefore gives
\begin{equation}\label{hlower}
 R^{-n-2s-1}\lesssim\Gamma(u).
\end{equation}
Finally, \eqref{eq:correction-norm} and Lemma~\ref{eqf} imply
\begin{align*}
 \|\rho\|_{H^s}
 &\le\|w_\sigma\|_{H^s}+\|\rho-w_\sigma\|_{H^s}\\
 &\lesssim\|f\|_{H^{-s}}+\Gamma(u)\\
 &\lesssim\Phi(R^{-n-2s-1})+\Gamma(u)
 \lesssim\Phi(\Gamma(u)).
\end{align*}
The last step uses \eqref{hlower}, monotonicity of $\Phi$ near zero,
$\Phi(ct)\simeq\Phi(t)$ for fixed $c>0$, and $t\lesssim\Phi(t)$.
This proves all three regimes, including the logarithmic endpoint.
\end{proof}

\section{The proof of Theorem~\ref{sharp}}\label{sec4}

We construct examples for every fixed number of peaks and every admissible
exponent. Throughout this section, choose
\[
    y_{i,R}=(i-1)Re_1,\qquad
    Q_{i,R}=Q(\cdot+y_{i,R}),\qquad
    \sigma_R=\sum_{i=1}^m Q_{i,R}.
\]
Let $f_R=\sigma_R^{1+\alpha}-\sum_iQ_{i,R}^{1+\alpha}$, and let
$\mathscr F_R$ be the span of the translation modes of this configuration.
In particular,
\[
    A(R)\simeq\Phi(R^{-n-2s-1}),\qquad
    R^{-n-2s-1}=o(A(R)),
\]
where
\begin{equation*}
A(R):=\begin{cases}
        R^{-(n+2s)},
        &\displaystyle \alpha>\alpha_* ,\\[3pt]
        R^{-(n+2s)}(\log R)^{\frac{1}{2}},
        &\displaystyle \alpha=\alpha_*,\\[3pt]
        R^{\frac{n}{2}-(1+\alpha)(n+2s)},
        &\displaystyle 0<\alpha<\alpha_*.
    \end{cases}
\end{equation*}
\begin{proposition}\label{prop:sharp-profiles}
For sufficiently large $R$, set
\[
 \rho_R=w_{\sigma_R},\qquad U_R=\sigma_R+\rho_R,
\]
where the correction is given by Lemma~\ref{lem:corrected-configuration}.
Then $\rho_R\in\mathscr F_R^\perp$,
\[
 P_{\mathscr F_R^\perp}L^{-1}\mathcal E(U_R)=0,
\]
and
\begin{equation}\label{eq:sharp-profile-estimates}
 \|\rho_R\|_{H^s}\simeq A(R),\qquad
 \Gamma(U_R)\simeq R^{-n-2s-1}.
\end{equation}
Consequently, $\|\rho_R\|_{H^s}\simeq\Phi(\Gamma(U_R))$.
\end{proposition}

\begin{proof}
Lemma~\ref{eqf} gives $\|f_R\|_{H^{-s}}\simeq A(R)$.
Existence, orthogonality, and the upper bound for $\rho_R$ follow from
Lemma~\ref{lem:corrected-configuration}; the residual estimate is
Lemma~\ref{lem:corrected-force}.
For the remaining lower bound, the uniformly invertible Gram matrix
and Lemma~\ref{lem:fQ} give
\[
 \|P_{\mathscr F_R}L^{-1}f_R\|_{H^s}
 \lesssim
 \left(\sum_{i=1}^m\sum_{j=1}^n
       \left|\int f_R\partial_{x_j}Q_{i,R}\,dx\right|^2\right)^{1/2}
 \lesssim R^{-n-2s-1}.
\]
Equation~\eqref{eq:correction-norm} therefore implies
\begin{align*}
 \|\rho_R\|_{H^s}
 &\simeq\|P_{\mathscr F_R^\perp}L^{-1}f_R\|_{H^s}\\
 &\ge\|f_R\|_{H^{-s}}-\|P_{\mathscr F_R}L^{-1}f_R\|_{H^s}
 \gtrsim A(R),
\end{align*}
since $R^{-n-2s-1}=o(A(R))$.
The last assertion follows from $A(R)\simeq\Phi(R^{-n-2s-1})$.
\end{proof}

We next verify that the correction measures the distance to all sums of
$m$ translates, rather than only to the particular configuration
$\sigma_R$.

\begin{lemma}\label{lem:sharp-distance}
For $U_R=\sigma_R+\rho_R$ as above,
\[
    d_m(U_R)\simeq\|\rho_R\|_{H^s}.
\]
\end{lemma}

\begin{proof}
The upper bound follows by choosing $z_i=y_{i,R}$. For the lower bound,
it is enough to consider configurations satisfying
\[
    \left\|U_R-\sum_iQ(\cdot+z_i)\right\|_{H^s}
       \le2\|\rho_R\|_{H^s}.
\]
For such configurations,
\[
    \left\|\sigma_R-\sum_iQ(\cdot+z_i)\right\|_{H^s}
       \le3\|\rho_R\|_{H^s}\to 0.
\]
We first show that, after a permutation,
$z_i=y_{i,R}+a_i$ with $\max_i|a_i|\to0$, uniformly among the
configurations under consideration. This can be checked along any
sequence $R\to\infty$. Near each center $-y_{i,R}$, the translated
function $\sigma_R(\cdot-y_{i,R})$ tends locally to $Q$.
If all the competing centers escaped that center, their translated sum
would tend locally to zero, contradicting the preceding $L^2$ convergence.
Thus every original center has a competing center at bounded distance.
Since the original centers separate and there are exactly $m$ competing
centers, this gives a bijective matching for large $R$.
After extracting a subsequence, each bounded displacement tends to some
$a_i$. Local convergence then gives $Q(\cdot+a_i)=Q$, hence $a_i=0$
by the strict radial decrease of $Q$. A contradiction argument gives the
asserted uniform convergence of the displacements.

Since $\partial_{x_j}Q\in H^s$, translation is continuously
differentiable as a map into $H^s$. Therefore
\[
    \left\|\sum_iQ(\cdot+y_{i,R}+a_i)-\sigma_R
       -\sum_i a_i\cdot\nabla Q_{i,R}\right\|_{H^s}
    =o\left(\Bigl(\sum_i|a_i|^2\Bigr)^{1/2}\right)
\]
as $\max_i|a_i|\to0$, uniformly in $R$.
Uniformity follows from translation invariance and the fixed number of peaks.
The Gram matrix bound gives
\[
    \left\|\sum_i a_i\cdot\nabla Q_{i,R}\right\|_{H^s}
       \ge c\Bigl(\sum_i|a_i|^2\Bigr)^{1/2}.
\]
Thus, for large $R$, the Taylor remainder is at most half the norm of
the linear term. Since $\rho_R$ is orthogonal to that term,
\begin{align*}
    \|U_R-\sum_iQ(\cdot+z_i) \|_{H^s}
    &\ge\left\|\rho_R-\sum_i a_i\cdot\nabla Q_{i,R}\right\|_{H^s}
       -\frac12\left\|\sum_i a_i\cdot\nabla Q_{i,R}\right\|_{H^s}\\
    &\ge\frac{1}{2}
       \left(\|\rho_R\|_{H^s}^2+
       \left\|\sum_i a_i\cdot\nabla Q_{i,R}\right\|_{H^s}^2\right)^{\frac{1}{2}}
     \ge\frac{1}{2}\|\rho_R\|_{H^s}.
\end{align*}
This proves the lower bound without assuming that the infimum is attained.
\end{proof}

\begin{proof}[Proof of Theorem~\ref{sharp}]
Proposition~\ref{prop:sharp-profiles} and Lemma~\ref{lem:sharp-distance}
give the desired examples among real-valued functions. We show that
their positive parts have the same properties. Set
\[
    u_R=U_R^+.
\]
Lipschitz truncation preserves $H^s$ for $0<s<1$. The scalar inequality
\[
    (a-b)(-a^-+b^-)\ge|a^--b^-|^2
\]
and the Gagliardo representation of the fractional energy imply
\[
    \|U_R^-\|_{H^s}^2\le\langle LU_R,-U_R^-\rangle.
\]
Testing the residual $\mathcal E(U_R)$ against $-U_R^-$ thus gives
\[
    \|U_R^-\|_{H^s}^2
       \le\Gamma(U_R)\|U_R^-\|_{H^s}
          +\int (U_R^-)^{\alpha+2}\,dx.
\]
Because $\sigma_R\ge0$, we have $U_R^-\le|\rho_R|$. H\"older's inequality
and Sobolev embedding yield
\begin{align*}
    \int (U_R^-)^{\alpha+2}\,dx
       &\le\int|\rho_R|^\alpha (U_R^-)^2\,dx\\
       &\le\|\rho_R\|_{L^{\alpha+2}}^\alpha\|U_R^-\|_{L^{\alpha+2}}^2
       \lesssim\|\rho_R\|_{H^s}^\alpha\|U_R^-\|_{H^s}^2.
\end{align*}
Since $\|\rho_R\|_{H^s}\to0$, absorption gives
\begin{equation}\label{eq:u-minus-decay}
    \|U_R^-\|_{H^s}\lesssim\Gamma(U_R)\lesssim R^{-n-2s-1}.
\end{equation}
The signed nonlinearity satisfies
$|U_R|^\alpha U_R=u_R^{1+\alpha}-(U_R^-)^{1+\alpha}$, whence
\[
    \mathcal E(u_R)-\mathcal E(U_R)=L(U_R^-)-(U_R^-)^{1+\alpha}.
\]
Using the dual Sobolev embedding, we obtain
\begin{equation}\label{eq:residual-upper-bound}
    \Gamma(u_R)
       \le\Gamma(U_R)+\|U_R^-\|_{H^s}
          +C\|U_R^-\|_{H^s}^{1+\alpha}
       \lesssim R^{-n-2s-1}.
\end{equation}

The distance to any subset of a normed space is a $1$-Lipschitz function.
Lemma~\ref{lem:sharp-distance} and \eqref{eq:u-minus-decay} therefore give
\begin{align*}
    d_m(u_R)
    &\ge c\|\rho_R\|_{H^s}-\|U_R^-\|_{H^s}\\
    &\ge c A(R),
\end{align*}
for sufficiently large $R$, since $R^{-n-2s-1}=o(A(R))$.
Also,
\[
    \|u_R-\sigma_R\|_{H^s}
       \le\|\rho_R\|_{H^s}+\|U_R^-\|_{H^s}
       \lesssim A(R)\to 0.
\]
The defect $\Gamma(u_R)$ is positive: otherwise
Theorem~\ref{thm-quantitative} would give zero distance to the manifold,
contradicting the preceding lower bound. Finally, $\Phi$ is increasing
near zero and $\Phi(ct)\simeq\Phi(t)$ for fixed $c>0$. Hence
\eqref{eq:residual-upper-bound} implies
\[
    \Phi(\Gamma(u_R))\lesssim\Phi(R^{-n-2s-1})\simeq A(R).
\]
Combining these estimates proves the asserted sharp lower bound for
nonnegative functions, for every fixed $m\ge2$ and every admissible
$\alpha$.
\end{proof}

\section*{Declarations}
On behalf of all authors, the corresponding author states that there
is no conflict of interest. Our manuscript has no associated data.

\section*{Acknowledgements}
Hua Chen is supported by the National Natural Science Foundation of China
(Grant Nos.~12571249, 12221001 and 12131017) and the National Key R\&D
Program of China (Grant No.~2022YFA1005602).
Xin Liao is supported by the National Natural Science Foundation of China
(Grant Nos.~12601372 and 12671239).
The research of Yun-Lu Fan is supported by the Postdoctoral Fellowship
Program of the China Postdoctoral Science Foundation
(Grant No.~GZC20261648).

\end{document}